\documentclass[letter,11pt]{amsart}
\usepackage{latexsym}
\usepackage{graphics}
\usepackage{amssymb}
\usepackage{amsmath}
\usepackage{epsfig}
\usepackage{tikz}%%
\usepackage{amscd}
\usepackage[all]{xy}

\newcommand{\ad}{A^{\nabla}}

\newcommand{\R}{\mbox {$\mathbb R $}}

\newtheorem{theorem}{Theorem}[section]
\newtheorem{lemma}[theorem]{Lemma}
\newtheorem{proposition}[theorem]{Proposition}
\newtheorem{corollary}[theorem]{Corollary}

\theoremstyle{definition}
\newtheorem{definition}[theorem]{Definition}
\newtheorem{example}[theorem]{\sc Example}

\theoremstyle{remark}

\begin{document}
  \vspace{10mm}

\begin{abstract}
In this note, we consider a triple construction $(\ad;\star,\epsilon(0))$ on a $d$-algebra $(A;\ast,0)$  and investigate some of  their properties. Applying this construction to a $d$-transitive $d$-algebra, we show that $(\ad; <)$ is a poset, which induces a $BCK$-algebra.
\end{abstract}

\title[A triple construction on $d$-algebras]{A triple construction on $d$-algebras }
\bigskip
\author{Hiba F. Fayoumi and  Akbar Rezaei}
\maketitle
\bigskip

\renewcommand{\thefootnote}{}


\section{Introduction}

Y. Imai and K. Is\'eki introduced two significant classes of abstract algebras: $BCK$-algebras and $BCI$-algebras \cite{Is,IsTa}. $BCK$-algebras have notable connections with various fields. For instance, D. Mundici \cite{Mu} proved that $MV$-algebras are categorically equivalent to bounded commutative $BCK$-algebras, while J. Meng \cite{Me} established that implicative commutative semigroups are equivalent to a class of $BCK$-algebras. Joseph Neggers and Hee Sik Kim later introduced the concept of $d$-algebras \cite{NK}, a useful generalization of $BCK$-algebras \cite{Is}, and explored numerous relationships between $d$-algebras and $BCK$-algebras, showing that $d$-algebras are among the least associative algebras.\\
Young Chan Lee and Hee Sik Kim further estimated the number of $d^*$- subalgebras of order $n$ in a $d$-transitive $d^*$-algebra \cite{LK}. Paul J. Allen \cite{A} expanded this area by constructing a large class of $d$-algebras through the use of constructive function triples on real numbers and integral domains. Sun Shin Ahn and Young Hee Kim \cite{AK} developed implicative and commutative $d$-algebras that are not $BCK$-algebras, proving that these are indeed generalizations of $BCK$-algebra concepts. Paul J. Allen et al. \cite{AKN} introduced a method for constructing $d$-algebras via deformation functions, which are distinct from $BCK$-algebras. Keum Sook So and Young Hee Kim \cite{SK} explored the conditions under which mirror algebras become $d$-algebras (or $d^*$-algebras).\\
The motivation behind this study lies in extending the properties of $d$-algebras as logical algebras. We examine a triple $(\ad;\star,\epsilon(0))$ on a $d$-algebra $(A;\ast,0)$ and investigate their related properties. Finally, we show that for any $d$-transitive $d$-algebra, ($A$ forms a poset, thus inducing a $BCK$-algebra).\\

\section{Preliminaries}
A $d$-{\it algebra} (\cite{NK}) is a non-empty set $X$ with a constant
0 and a binary operation ``$\ast$" satisfying the following
axioms:%\par
\begin{itemize}
\item[(I)] $x*x = 0$,
\item[(II)] $0*x = 0$,
\item[(III)] $x*y =0$ and $y*x= 0$ imply $x=y$
%\par\noindent
%for all $x, y$ in $X$.
\end{itemize}
%\itemitem{(I)} $x*x = 0$,
%\itemitem{(II)} $0*x = 0$,
%\itemitem{(III)} $x*y =0$ and $y*x= 0$ imply $x=y$
%\par\noindent
for all $x, y$ in $X$.%\par
\\
A $d$-algebra $X$ is said to be a $d^{\ast}$-algebra (\cite{LK}) if it satisfies the following axiom: for all $x,y\in X$
\begin{itemize}
\item[(IV)] $(x\ast y)\ast x=0$.
\end{itemize}
%\medskip
A $BCK$-algebra is a $d$-algebra $(X;*,0)$ satisfying the
following additional axioms:%\par
\begin{itemize}
\item[(V)]  $((x*y)*(x*z))*(z*y)=0$,
\item[(VI)] $(x*(x*y))*y=0$
\end{itemize}
for all $x, y, z$ in $X$.
\par

In $X$, we can define a binary relation ``$\leq$"  by $x\leq y$ if and only if $x\ast y=0$. Note that if $X$ is a $d$-algebra with $x\ast 0=0$ for any $x\in X$, then $x=0.$
\\
 It is known that if $(X;*,0)$ is a $BCK/BCI$-algebra then
\begin{itemize}
 \item[(VII)] $(x*y)*(z*y)\leq x*z$,
 \item[(VIII)] $x*(x*y)\leq y$,
 \item[(IX)] $x*0 = x$,
\item[(X)] $x*y\leq x$
 \end{itemize}
 for any $x,y,z\in X$.

Clearly, a $BCK$-algebra is a $d^*$-algebra, but the converse need not be true.
\begin{example} \cite{NK} \label{1.1}
   Let $\R$ be the set of all real
numbers and define $x*y := x\cdot(x-y)$, $x,y\in \R$, where
$``\cdot"$ and $``-"$ are ordinary product and subtraction of
real numbers. Then $x*x = 0, 0*x = 0, x*0 = x^2$. If $x*y = y*x =
0$, then $x(x-y) = 0$ and $x^2 = xy$, $y(y-x) = 0$, $y^2 = xy$.
Thus if $x = 0$, $y^2 = 0$, $y=0$; if $y=0$, $x^2 = 0$, $x = 0$
and if $xy\not = 0$, then $x=y$. Hence $(\R;*,0)$ is a $d$-algebra,
but not a $d^*$-algebra, since $(2*0)*2 \not = 0$.
\end{example}

\begin{definition}\label{1.2}
{\rm \cite{NK}} {\rm Let $(X;*,0)$ be a
$d$-algebra and $x\in X$. Define $x\ast X:=\{x*a\vert a\in X\}$.
$X$ is said to be {\it edge} if for any $x$ in $X$, $x*X= \{x,
0\}$.}
\end{definition}

 If $(X,\leq)$ is an ordered set (poset), then the operation $*$
on $X$ given by $x*y = 0$ iff $x\leq y$ and $x*y = x$ otherwise
defines a $BCK$-algebra. On the other hand, from our viewpoint it
has the ``edge" property. Although edge $d$-algebras are not
$BCK$-algebras in general, they come close to being so as we note
below.

\begin{lemma}\label{1.3}
{\rm \cite{NK}} Let $(X;*,0)$ be an edge
$d$-algebra. Then $x*0= x$ for any $x\in X$.
\end{lemma}

\begin{proposition}\label{1.4}
{\rm \cite{NK}} If $(X;*,0)$ is an edge
$d$-algebra, then the condition {\rm(V)} holds.
\end{proposition}

\begin{definition}\label{1.5}
{\rm A $d$-algebra $(X;*,0)$ is said to
be $d$-{\it transitive} if $x*z = 0$ and $z*y= 0$ imply $x*y=0$. }
\end{definition}

\begin{theorem}\label{1.6}
{\rm \cite{NK}} Let $(X; *,0)$ be a
$d$-transitive edge $d$-algebra. Then $(X;*,0)$  is a
$BCK$-algebra.
\end{theorem}

Both conditions, i.e., $d$-transitive and edge, are necessary for
a $d$-algebra of this type to be a $BCK$-algebra. Thus, arbitrary
$BCK$-algebras do not always have the edge property even if the
standard examples derived from posets do indeed possess it.

\begin{example}\label{1.7}
{\rm Let $X:=\{0, 1, 2, \cdots \}$ and the binary operation
 ``$\ast$" be defined  as follows:}%\par
\[x*y:=\left\{ \begin{array}{ll}  0  & \text{\rm if} \  \ x\leq y,
\\1 & \text{\rm otherwise}.
\end{array}
\right.\]

\noindent {\rm Then $x*z = 0$, $z*y=0$ implies $x\leq z, z\leq y$
and in particular $x\leq y$, i.e., $x*y = 0$ also. Furthermore,
$x*x = 0, 0*x = 0$ and $x*y = y*x = 0$ if $x\leq y, y\leq x$,
whence $x=y$. Thus, the algebra $(X;*,0)$ is a $d$-transitive
non-edge $d$-algebra. Also, $(2*(2*0))*0 = (2*1)*0 = 1*0 = 1\ne 0$,
so that  $(X;*,0)$ is not a $BCK$-algebra.    }
\end{example}

\section{The main results}
Given a $d$-algebra $(A; \ast, 0)$, we construct the $d$-triple $(A^{\nabla}; \star, \epsilon(0))$, where 
\begin{itemize}
\item[(i)] $A^{\nabla}$ is the normalizer of  $(A; \ast, 0)$, such that\\ $ \ad =\{(a,b,c) \,\vert\, a\ast b = 0 = b\ast c, a,b, c, \in A\}$
\item[(ii)]  $\star: A^{\nabla}\times \ad\to \ad$ is defined as  $(a,b,c)\star(d,e,f) := (a*f, b*e, c*d)$ such that  $(a*f)*(b*e) = 0 = (b*e)*(c*d),$ for $a, b, c, d, e, f \in A.$
\item[(iii)] $\epsilon(x): A\to \ad$ is the triple projection of $x$, such that $\epsilon(x)=(x,x,x)$, for any $x\in A.$
\end{itemize}

Note that $\ad$ is a non-empty set, since $0\in A$ and $0\ast 0=0$, we get  $\epsilon(0)=(0,0,0)\in \ad,$ and so $\epsilon(0)\in \ad.$ Furthermore, by using (I), we have  $\epsilon(x)\in\ad$ and $\epsilon(x)\star\epsilon(x)=\epsilon(0),$ for any $x\in A$.  Also,  $ \epsilon(0)\star(a,b,c) = \epsilon(0)$, for any
  $(a,b,c)\in \ad$. If $A$ is an edge $d$-algebra
  then $(a,b,c) \star  \epsilon(0) = (a,b,c)$.
  Actually, $(\ad; \star, \epsilon(0))$ is only a partial algebra
  since the operation $\star$ may not be defined everywhere.

% \example{Example 3.1}
\begin{example}\label{2.1}
(i)  Consider the $d$-algebra $(A;*,0)$, where $*$ is defined by Table 1 below.
\\

\begin{table}
					\small
					\caption{$d$-algebra $(A;*,0)$}\label{eqtable}
	\begin{center}
	    
	\begin{tabular}{c|ccccc}
		\hline
		$\ast$  & 0 & 1 & 2 & 3 & 4\\
		\hline
   		0  & 0 & 0 & 0 & 0 & 0\\
		1  & 1 & 0 & 2 & 0 & 4\\
		2  & 2 & 2 & 0 & 3 & 0\\
        3  & 3 & 3 & 3 & 0 & 3\\
 4  & 4 & 4 & 4 & 1 & 0\\
		\hline
\end{tabular}
\end{center}
\end{table}

 It is easy to see  that
 $(0, 2,4)\star(0, 1,3) = (0,2,4)$, while $(0,1,3)\star(0,2,4)$ is not
 defined since $(1*2)*(3*0) = 2*3 = 3 \not = 0$.
\\
(ii) Consider the $d$-algebra $(\R,\ast,0)$ which is given in Example \ref{1.1}. By calculations, we see that $\R^{\bigtriangledown}=\{\epsilon(x), (0,x,x), (0,0,x)|x\in\R\}.$
\end{example}

 Note that
 $(\ad; \star, \epsilon(0))$ does not satisfy axiom (I), since
 $(0,2,4)\star(0,2,4) = (0,0, 4)$ in Example \ref{2.1} (i). Furthermore, for every $(a,b,c)\in\ad$, may be $(c,b,a)\not\in\ad$, in general, e.g.,  in Example \ref{2.1} (ii), $(0,2,2)\in\R^{\bigtriangledown}$, but $(2,2,0)\not\in\R^{\bigtriangledown}$, since $2\ast 2=0$ but  $2\ast 0=2.$ %3.1.

\begin{proposition}
Let $(A;\ast,0)$ be a $d$-algebra {\rm(}resp. $d^*$-algebra{\rm)}. %Take $S:=\{\epsilon(x)|x\in A\}.$ 
Then $(S;\star,\epsilon(0))$ is a $d$-algebra  {\rm(}resp. $d^*$-algebra{\rm)}, where   $S=\{\epsilon(x)|x\in A\}.$
\end{proposition}

\begin{theorem}\label{2.2}
 If $(A;*,0)$ is a $BCK/BCI$-algebra, then $(\ad; \star, \epsilon(0))$
 is an algebra, i.e., $\star$ is a binary operation with domain
 $\ad\times \ad$.
\end{theorem}
\begin{proof}

Let $(a,b,c), (d,e,f)\in \ad$. Then $a*b = b*c = 0$ and $d*e = e*f = 0$.
 Since $A$ is a $BCK/BCI$-algebra, we have $((a*f)*(a*e))*(e*f) = 0$.
 It follows from (IX) that $(a*f)*(a*e)=0$.
 Applying (VII) we have $(a*e)*(b*e)\leq a*b = 0$
 and so $(a*e)*(b*e) =0$ by (IX).
 Using (VII), we obtain $(a*f)*(b*e) = 0$.
 Similarly, we obtain $(b*e)*(c*d) =0$.
 This means that $(a,b,c)\star(d,e,f)\in \ad$, which proves the theorem.

\end{proof}

\begin{proposition}\label{2.3}
 If $(A;*,0)$ is a $d$-transitive $d$-algebra, then $(a,b,c)\star (a,b,c)\in \ad$, for any $(a,b,c)\in \ad$.
\end{proposition}
\begin{proof}
Assume  $(a,b,c)\in \ad$. Then $a*b = 0 =b*c$. Since $A$ is $d$-transitive, $a*c = 0$.
 Hence $(a,b,c)\star(a,b,c)= (a*c, 0, c*a) = (0,0, c*a)\in \ad$.

\end{proof}

 The converse of Proposition \ref{2.3} is also true for the cases of
 edge $d$-algebras or $BCK/BCI$-algebras.
 Especially if it is an edge $d$-algebra then by Theorem \ref{1.6}, it should be a $BCK$-algebra.

\begin{proposition}\label{2.4}
 Let $(A;*,0)$ be a $d$-algebra.
 If $(a,b,c)\star(d,e,f) = (d,e,f)\star(a,b,c) = \epsilon(0)$,
 then $(a,b,c) = (d,e,f) .$%= \epsilon(b)$.
\end{proposition}
\begin{proof}
Let $(a,b,c)\star(d,e,f) = (d,e,f)\star(a,b,c) = \epsilon(0)$. Then
 $a*f = b*e = c*d = 0$ and $d*c = e*b = f*a = 0$.
 Since $A$ is a $d$-algebra, we obtain $a=f, b=e, c=d$.
 Hence $(f,e,d)= (a,b,c) \in\ad$. This means that $f*e = e*d = 0$. Since $(d,e,f)\in \ad$, we have $f = e=d$. Thus, $(a,b,c) = (d,e,f).$% = \epsilon(b)$.
 
\end{proof}

\begin{theorem}\label{2.5}
 Let $(A;*,0)$ be an edge $d$-algebra and
 $(a,b,c)\star(d,e,f)\in\ad$.
 Then $(a,b,c)\star(d,e,f)= \epsilon(0)$ if and only if $c*d = 0$.

\end{theorem}
\begin{proof}
 Suppose $c*d = 0$.
 Since $(a,b,c)\star(d,e,f)\in\ad$ and $A$ is an edge $d$-algebra,
 $0=(b*e)*(c*d) = (b*e)*0 = b*e$.
 Hence $a*f = (a*f)*0 = (a*f)*(b*e) = 0$, proving
 $(a,b,c)\star(d,e,f)= \epsilon(0)$.
 The converse is trivial, and we omit it.
 
\end{proof}

 We write $(a,b,c)<(d,e,f)$ provided
 $(a,b,c)\not =(d,e,f)$ and $(a,b,c)\star(d,e,f)=\epsilon(0)$, e.g., in Example \ref{2.1} (ii), $(0,0,3)<(0,3,3)$, but $(0,0,2)\not< (0,3,3)$, since \[(0,0,2)\star(0,3,3)=(0\ast 3,0\ast 3,2\ast0)=(0,0,4)\ne\epsilon(0).\]

 Note that %$(0,0,0)\star(a,b,c) =(0,0,0)$, 
since $\epsilon(0)\star (a,b,c)=\epsilon(0)$, we get $\epsilon(0)<(a,b,c)$ when $(a,b,c)\not =\epsilon(0)$, i.e.,
 $\epsilon(0)$ is the unique minimal element in this ``order". Applying Proposition \ref{2.5}, we get ``$<$" is antisymmetric.
 \\
 Given an element $(a,b,c)\in \ad$ we define
 $d(a,b,c):=c*a$, and we call it the {\it diameter} of $(a,b,c)$.
\\
Note that, let $(A;\ast,0)$ be a $d$-algebra. Using (I) and (II), we get  $d(\epsilon(x))=d(x,x,x)=x\ast x=0$, $d(x,y,x)=x\ast x=0$ and $d(x,y,0)=0\ast x=0$,  for any $x, y\in A$. Furthermore, if $(A;\ast,0)$ is a $BCK$-algebra,  then $d(0,y,z)=z\ast 0=0$, for all $y,z\in A.$
\\
In Example \ref{2.1} (ii), $d(0,0,x)=d(0,x,x)=x\ast 0=x\cdot(x-0)=x^2$, for every $x\in\R.$ Hence, in such case the diameters $(0,0,x)$ and $(0,x,x)$ are zero if and only if $x=0.$

\begin{proposition}\label{2.7}
 If $(A;*,0)$ is a $d$-algebra and $d(a,b,c)=d(c,b,a)=0$, then ``$<$" is reflexive.

\end{proposition}
\begin{proposition}\label{2.6}
 If $(A;*,0)$ is a $d$-transitive $d$-algebra, then $d(a,b,c) = 0$ if and only if $(a,b,c) =\epsilon(a)$.

\end{proposition}

The following proposition is easy to prove.

\begin{proposition}\label{2.8}
 If $(A;*,0)$ be $d$-transitive  $d$-algebra,
 then ``$<$" is transitive.
\end{proposition}

\begin{theorem}\label{2.9}
Let $(A;\ast,0)$ be a  $d$-transitive  $d$-algebra and $d(a,b,c)=0$, for any $(a,b,c)\in\ad$.
 Then $(\ad;<)$ is a poset. 
\end{theorem}

\begin{corollary}\label{2.10}
If  $(A;\ast,0)$ is a  $d$-transitive $d$-algebra and $(a,b,c)=\epsilon(a),$ %$(c,b,a)=\epsilon(c)$, 
for any $(a,b,c)\in\ad$, then $(\ad;<)$ is a poset. 
\end{corollary}

\begin{corollary}\label{2.11}
If $(A;\ast,0)$ be a  $d$-transitive  $d$-algebra and $d(a,b,c)=0$, for any $(a,b,c)\in\ad$, then 
$(\ad;<)$ induces a $BCK$-algebra. 
\end{corollary}

\begin{corollary}\label{2.12}
If $(A;\ast,0)$ be a  $d$-transitive  $d$-algebra and $(a,b,c)=\epsilon(a),$ for any $(a,b,c)\in\ad$, then $(\ad;<)$ induces a $BCK$-algebra. 
\end{corollary}

Consider the $d$-transitive $d$-algebra $(X;\ast,0)$ given in Example \ref{1.7}, which was not a $BCK$-algebra. Now, applying Corollary \ref{2.12}, we can see that $(X^{\bigtriangledown};<)$ induced a $BCK$-algebra.

\section*{Acknowledgment} The authors express their gratitude to Professor Hee Sik Kim for his contribution to this article and for his many valuable suggestions.
				
				\hspace{1in}

\bigskip

\footnotesize{\textsc{Hiba F. Fayoumi, Department of Mathematics and Statistics,
University of Toledo, Toledo, OH 43606-3390, U. S. A.}
\par \textit{E-mail address}: \texttt{hiba.fayoumi@UToledo.edu}}
\medskip

 \footnotesize{\textsc{ Akbar Rizae, Department of Mathematics, Payame Noor University, P.O. Box 14395-469, Tehran, Iran}
\par \textit{E-mail address}: \texttt{rezaei.pnuk@gmail.com}}
\medskip

\bigskip
%%%%%%%%%%%%%%%%%%%%%%%%%%%%%%%%%%%%%%%%%%%%%%%%%%%%%%%%%%%%%

\end{document}